\documentclass[11pt, english, draft]{article}
\usepackage{amssymb,amsmath}
\title{\bf On Recognition by Order and Degree Pattern of Finite Simple Groups }
\author{ {\bf B. Akbari} and {\bf A. R. Moghaddamfar}\\[0.1cm]
{\em Department of Mathematics, K. N. Toosi
University of Technology,}\\
 {\em P. O. Box $16315$-$1618$, Tehran, Iran}\\[0.1cm]
 {\em and} \\[0.1cm]
{\em E-mails}: {\tt moghadam@kntu.ac.ir} and {\tt
moghadam@ipm.ir}}

\newenvironment{proof}{\noindent {\em {Proof}}.}{$\square$
\medskip}
\newtheorem{definition}{Definition}[section]
\newtheorem{corollary}{Corollary}[section]

\newtheorem{theorem}{Theorem}[section]
\newtheorem{proposition}{Proposition}[section]
\newtheorem{remark}{Remark}[section]
\newtheorem{lm}{Lemma}[section]

\newtheorem{problem}{Problem}[section]
\begin{document}
\maketitle
\begin{abstract}
\noindent Let ${\rm GK}(G)$ be the prime graph associated with a
finite group $G$ and $D(G)$ be the degree pattern of $G$. A
finite group $G$ is said to be $k$-fold OD-characterizable if
there exist exactly $k$ non-isomorphic groups $H$ such that
$|H|=|G|$ and $D(H)=D(G)$. A 1-fold OD-characterizable group is
simply called OD-characterizable. The purpose of this paper is
threefold. First, it provides the reader with a few useful and
efficient tools on OD-characterizability of finite groups. Second,
it lists a number of such simple groups that have been already
investigated. Third, it shows that the simple groups $L_6(3)$ and
$U_4(5)$ are OD-characterizable, too.\\[0.3cm]
{\bf Keywords}: prime graph, degree pattern, simple group.
\end{abstract}
\renewcommand{\baselinestretch}{1.1}
\def\thefootnote{ \ }
\footnotetext{{\em $2010$ Mathematics Subject Classification}:
20D05, 20D06.}
\section{Introduction and  Previous Results} In this paper, we will consider only finite groups.
The set of elements order of a finite group $G$ is denoted by
$\omega(G)$ and called the spectrum of $G$ (see \cite{za-1}).
This set is closed and partially ordered by the divisibility
relation; therefore, it is determined uniquely from the subset
$\mu(G)$ of all maximal elements of $\omega(G)$ with respect to
divisibility. For a natural number $n$ we denote by $\pi(n)$ the
set of prime divisors of $n$ and put $\pi(G)=\pi(|G|)$.

There are a lot of ways to associate a graph to a finite group.
One of well-known graphs associated with a finite group $G$ is
called {\em prime graph} (or {\em Gruenberg-Kegel graph}) of $G$
and denoted by ${\rm GK}(G)$, which is a simple undirected graph
and construct as follows. The vertices are all prime divisors of
$|G|$ and two distinct vertices $p$ and $q$ are adjacent if and
only if $pq\in \omega(G)$. The number of connected components of
${\rm GK}(G)$ is denoted by $s(G)$, and the connected components
are denoted by $\pi_i=\pi_i(G)$, with $1\leq i\leq s(G)$. When the
group $G$ has even order, we will assume that $2\in \pi_1(G)$.
Finite simple groups with disconnected prime graph were described
in \cite{k, w}. A complete list of these groups (with corrected
misprints) can be found in \cite[ Tables 1(a)-1(c)]{mazurov-2002}.

In \cite{suz2}, Suzuki studied the structure of the prime graphs
associated with finite simple groups, and found the following
interesting result.
\begin{proposition}[Suzuki]\label{prop1} (\cite {suz2}, Theorem B)  Let $G$ be a
finite simple group whose prime graph ${\rm GK}(G)$ is
disconnected and let $\Delta$ be a connected component of ${\rm
GK}(G)$ whose vertex set does not contain $2$. Then $\Delta$ is a
clique.
\end{proposition}

\begin{remark} By a clique we shall mean a complete graph and we write $K_n$ to
be the clique of order $n$. Let $\Gamma=(V, E)$ be a simple graph.
Note that, in Proposition \ref{prop1}, it does not require $G$ to
be simple. Moreover, it shows that the prime graph of $G$ has the
following structure:
$${\rm GK}(G)={\rm GK}[\pi_1]\oplus K_{n_2}\oplus \cdots\oplus K_{n_s},$$
where ${\rm GK}[\pi_1]$ denotes the induced subgraph ${\rm
GK}(G)[\pi_1]$, $n_i=|\pi_i|$ and $s=s(G)$. We recall that the
induced subgraph ${\rm GK}(G)[\pi_1]$ has the vertex set $\pi_1$
and it contains all edges of ${\rm GK}(G)$ which join vertices in
$\pi_1$.
\end{remark}

The {\em degree} $\deg_G(p)$ of a vertex $p\in \pi(G)$ is the
number of edges incident on $p$ in ${\rm GK}(G)$. If
$\pi(G)=\{p_1, p_2, \ldots, p_k\}$ with $p_1< p_2< \cdots< p_k$,
then we put
\[ {\rm D}(G):=\big(\deg_G(p_1), \deg_G(p_2), \ldots, \deg_G(p_k)\big), \]
and call this $k$-tuple the {\em degree pattern of $G$}.  We
denote by $h_{\rm OD}(G)$ the number of isomorphism classes of
finite groups $H$ such that $|H|=|G|$ and ${\rm D}(H)={\rm
D}(G)$. Obviously, $1\leqslant h_{\rm OD}(G)<\infty$ for any
finite group $G$. In terms of function $h_{\rm OD}(\cdot)$, the
groups $G$ are classified as follows:

\begin{definition} A finite group $G$ is called a {\em $k$-fold
OD-characterizable} group if $h_{\rm OD}(G)=k$. A $1$-fold
OD-characterizable group is simply called an {\em
OD-characterizable} group.
\end{definition}
The notation for sporadic and simple groups of Lie type is
borrowed from \cite{atlas}. Moreover, we use the notation
$\mathbb{A}_n$ to denote an alternating of degree $n$. There are
scattered results in the literature showing that many of simple
groups have the same prime graphs, for instance, it was
shown in \cite{hag}, \cite{vv}, \cite{zve} that each of the following sets consist of simple groups\\

\begin{tabular}{llll} (a) &  $\{S_{2n}(q), O_{2n+1}(q)\}$ & &
\cite[Proposition
7.5]{vv}\\[0.3cm]
(b) & $\{L_2(11), M_{11}\}$,  & & \cite[Theorem 3]{hag}\\[0.3cm]
(c) &  $\{\mathbb{A}_5, \mathbb{A}_6\}$, $\{\mathbb{A}_7, L_2(49),
U_4(3)\}$, $\{\mathbb{A}_9, J_2, S_6(2), O^+_8(2)\}$, &  &\cite[Theorem]{zve}\\[0.3cm]
(d) &  $\{\mathbb{A}_n, \mathbb{A}_{n-1}\}$, $n$ is odd, and $n$
and $n-4$ are composite,&  &  \cite[Theorem]{zve}.
\\[0.5cm]
\end{tabular}

\noindent have the same prime graphs and so the same degree
pattern. On the other hand, the symplectic group $S_{2n}(q)$ and
the orthogonal group $O_{2n+1}(q)$ have the same orders, too, that
is
$$|S_{2n}(q)|=|O_{2n+1}(q)|=\frac{1}{(2,q-1)}q^{n^2}\prod\limits^n_{i=1}(q^{2i}-1).$$
Therefore, in the case that $S_{2n}(q)\ncong O_{2n+1}(q)$ (that
is $q$ odd and $n>2$), we have
$$h_{\rm OD}(S_{2n}(q))=h_{\rm OD}(O_{2n+1}(q))\geq 2.$$
Notice that, we have $S_{2n}(2^m)\cong O_{2n+1}(2^m)$ and
$S_4(q)\cong O_{5}(q)$ (see for example \cite{carter}). There is
also another pair of non-isomorphic simple groups with the same
order, that is $L_4(2)\cong \mathbb{A}_8$ and $L_3(4)$. Actually,
about simple groups with the same order, we have the following
proposition (see \cite{shi}).

\begin{proposition}\label{shi-simple} Every finite simple group can be
determined up to isomorphism in the class of finite simple groups
by its order, except exactly in the following cases:
\begin{itemize}
\item[{(a)}] $L_4(2)\cong \mathbb{A}_8$ and $L_3(4)$  have the same order, i.e., $20160$,
\item[{(b)}] $O_{2n+1}(q)$ and $S_{2n}(q)$ have
the same order for $q$ odd, $n>2$.
\end{itemize}
\end{proposition}

As an immediate consequence of Proposition \ref{shi-simple} we
have the following corollary.

\begin{corollary}\label{cor1} If $G$ is a finite simple group, then there
exists at most one finite simple group, non-isomorphic to $G$,
with the same order and degree pattern as $G$. \end{corollary}

It was currently shown that many finite simple groups are
OD-characterizable or $2$-fold OD-characterizable. We have listed
these groups in Table 1.
\begin{center}
{\bf Table 1}. Some non-abelian simple groups $S$ with $h_{\rm
OD}(S)=1$ or $2$.\\[0.4cm]
$\begin{array}{l|l|c|l} \hline S & {\rm Conditions \ on} \ S&
h_{\rm OD}(S) & {\rm Refs.} \\ \hline
 \mathbb{A}_n & \ n=p, p+1, p+2 \ (p \ {\rm a \ prime})& 1 &  \cite{mz1}, \cite{mzd}    \\
 & \ 5\leqslant n\leqslant 100, n\neq 10   & 1 & \cite{hm}, \cite{kogani}, \cite{banoo-alireza},  \\
 & & & \cite{mz3}, \cite{zs-new2} \\
& \ n=106, \ 112 & 1 &      \cite{yan-chen}       \\
& \ n=10 & 2 &      \cite{mz2}       \\[0.2cm]
L_2(q) &  q\neq 2, 3& 1 &    \cite{mz1}, \cite{mzd},\\
& & &  \cite{zshi} \\[0.1cm]
L_3(q) &  \ |\pi(\frac{q^2+q+1}{d})|=1, \ d=(3, q-1) & 1 &   \cite{mzd} \\[0.2cm]
U_3(q) &  \ |\pi(\frac{q^2-q+1}{d})|=1, \ d=(3, q+1), q>5 & 1 &   \cite{mzd} \\[0.2cm]
L_3(9) & & 1 & \cite{zs-new4}\\[0.1cm]
U_3(5) &   & 1 &   \cite{zs-new5} \\[0.1cm]
L_4(q) &  \ q\leqslant 37  & 1 &   \cite{bakbari}, \cite{akbari-1}, \cite{amr} \\[0.1cm]
U_4(7) &   & 1 &   \cite{amr} \\[0.1cm]
L_n(2) & \ n=p \ {\rm or} \ p+1, \ {\rm for \ which} \ 2^p-1 \ {\rm is \ a \ prime} & 1 & \cite{amr} \\[0.2cm]
L_n(2) & \ n=9, 10, 11  & 1 &   \cite{khoshravi}, \cite{R-M} \\[0.1cm]
U_6(2) & & 1 & \cite{LShi} \\[0.1cm]
R(q) & \ |\pi(q\pm \sqrt{3q}+1)|=1, \ q=3^{2m+1}, \ m\geqslant 1 & 1 & \cite{mzd} \\[0.2cm]
{\rm Sz} (q) & \ q=2^{2n+1}\geqslant 8& 1 &   \cite{mz1}, \cite{mzd} \\[0.2cm]
B_m(q), C_m(q) &  m=2^f\geqslant 4,  \
|\pi\big((q^m+1)/2\big)|=1, \
 & 2 & \cite{akbarim}\\[0.2cm]
B_2(q)\cong C_2(q) &  \ |\pi\big((q^2+1)/2\big)|=1, \ q\neq
3 & 1 & \cite{akbarim}\\[0.2cm]
B_m(q)\cong C_m(q) &  m=2^f\geqslant 2, \ 2|q, \
|\pi\big(q^m+1\big)|=1, \ (m, q)\neq (2, 2) & 1 &
\cite{akbarim}\\[0.2cm]
B_p(3), C_p(3) &  |\pi\big((3^p-1)/2\big)|=1, \  p \ {\rm is \ an
\ odd \
prime}  & 2 & \cite{akbarim}, \cite{mzd}\\[0.2cm]
B_3(5), C_3(5) & & 2 & \cite{akbarim} \\[0.2cm]
C_3(4) & & 1 & \cite{moghadam} \\[0.1cm]
S &  \ \mbox{A sporadic simple group} & 1 & \cite{mzd} \\[0.1cm]
S &  \ \mbox{A simple group with} \ |\pi(S)|=4, \ \ S\neq \mathbb{A}_{10} & 1 & \cite{zs} \\[0.1cm]
S &  \  \mbox{A simple group with} \ |S|\leqslant 10^8, \ \ S\neq \mathbb{A}_{10}, \ U_4(2) & 1 & \cite{ls} \\[0.1cm]
S &  \  \mbox{A simple $C_{2,2}$- group} & 1 & \cite{mz1}\\[0.2cm]
\hline
\end{array}$
\end{center}
\vspace{0.5cm} According to the results in Table 1, $h_{\rm
OD}(\mathbb{A}_{10})=2$. In fact, another group with the same
order and degree pattern as $\mathbb{A}_{10}$ is
$\mathbb{Z}_3\times J_2$, that is
$$ |\mathbb{A}_{10}|=|\mathbb{Z}_3\times J_2|=2^{7}\cdot 3^{4}\cdot 5^{2}\cdot 7,
 \ \ \ \ D(\mathbb{A}_{10})=D(\mathbb{Z}_3\times J_2)=(2, 3, 2, 1).$$

Until recently, no example of simple group $M$ has been found
with $h_{\rm OD}(M)\geq 3$. So it would be reasonable to ask
whether it might be possible to find such a simple group. The
following question has been posed in \cite{mz1}:
\begin{problem} Is there a simple group which is $k$-fold
OD-characterizable for $k\geq 3$?
\end{problem}

It is worthwhile to point out that according to the Corollary
\ref{cor1}, if there exists a simple group, say $S$, with $h_{\rm
OD}(S)\geq 3$, then among non-isomorphic groups with the same
order and degree pattern as $S$, certainly there will be a
non-simple group.

In this paper especially we will concentrate on simple groups
$L_6(3)$ and $U_4(5)$. In fact, we will show that both groups are
OD-characteizable. It is worth noting that the prime graph of
$L_6(3)$ is
disconnected, while the prime graph of $U_4(5)$ is connected.\\[0.3cm]
{\bf Main Theorem.} \ {\em The simple groups $L_6(3)$ and
$U_4(5)$ are OD-characteizable.}

\section{OD-Characterizability of Simple Groups $L_6(3)$ and $U_4(5)$}
In what follows, we will show that the projective special linear
group $L_6(3)$ and the projective special unitary group $U_4(5)$
are uniquely determined by order and degree pattern.

\begin{lm} The following statements hold.
\begin{itemize}
\item[{\rm (a)}] If $L$ be the finite simple group $L_6(3)$, then

\subitem{\rm (a.1)} $|L|=2^{11}\cdot 3^{15} \cdot 5\cdot 7\cdot
11^2\cdot 13^2$;

\subitem{\rm (a.2)} $\mu(L)=\{182, 121, 120, 104, 80, 78, 36\}$;

\subitem{\rm (a.3)} ${\rm Out}(L)=\mathbb{Z}_2\times
\mathbb{Z}_2$ and $s(L)=2$;

\subitem{\rm (a.4)} $D(L)=(4, 3, 2, 2, 0, 3)$.

\item[{\rm (b)}] If $U$ be the finite simple group $U_4(5)$, then

\subitem{\rm (b.1)} $|U|=2^7\cdot 3^4\cdot 5^6\cdot 7\cdot 13$;

\subitem{\rm (b.2)} $\mu(U)=\{63, 60, 52, 24\}$;

\subitem{\rm (b.3)} ${\rm Out}(U)=\mathbb{Z}_2\times \mathbb{Z}_2$
and $s(U)=1$;

\subitem{\rm (b.4)} $D(U)=(3, 3, 2, 1, 1)$.
\end{itemize}
\end{lm}
\begin{proof} See \cite{atlas}. \end{proof}

{\small  \setlength{\unitlength}{4mm}
\begin{picture}(1,1)(-8,7)
\linethickness{0.3pt} %
\put(21,24){\circle*{0.5}}
\put(26,24){\circle*{0.5}}
\put(23.5,21.5){\circle*{0.5}}
\put(28.5,21.5){\circle*{0.5}}
\put(18.5,21.5){\circle*{0.5}}
\put(18.5,21.5){\line(1,1){2.5}}
\put(21,24){\line(1,0){5}}
\put(26,24){\line(1,-1){2.5}}
\put(21,24){\line(1,-1){2.5}}
\put(23.5,21.5){\line(1,1){2.5}}
\put(18.5,21.5){\line(1,0){10}}
\put(23.5,26.5){\circle*{0.5}}
\put(18.2,20.2){\small$5$}%
\put(20.8,24.7){\small$3$}%
\put(25.8,24.7){\small$13$}%
\put(28.2,20.2){\small$7$}%
\put(23.3,20.2){2}%
\put(23.1,27.2){11}%
\put(17,18){{\small \bf Fig. 1.} \ \ \mbox{\small The prime graph
of} $L_6(3)$.}
\put(21,15){\circle*{0.5}}
\put(26,15){\circle*{0.5}}
\put(23.5,12.5){\circle*{0.5}}
\put(28.5,12.5){\circle*{0.5}}
\put(18.5,12.5){\circle*{0.5}}
\put(18.5,12.5){\line(1,1){2.5}}
\put(21,15){\line(1,0){5}}
\put(26,15){\line(1,-1){2.5}}
\put(21,15){\line(1,-1){2.5}}
\put(23.5,12.5){\line(1,1){2.5}}
\put(18.2,11.2){\small$13$}%
\put(20.8,15.7){\small$2$}%
\put(25.8,15.7){\small$3$}%
\put(28.2,11.2){\small$7$}%
\put(23.3,11.2){5}%
\put(17,9){{\small \bf Fig. 2.} \ \ \mbox{\small The prime graph
of} $U_4(5)$.}
\end{picture}}

Let $\Gamma=(V, E)$ be a graph with vertex set $V$ and edge set
$E$. A set $I\subseteq V$ of vertices is said to be an
independent set of $\Gamma$ if no two vertices in $I$ are
adjacent in $\Gamma$. The independence number of $\Gamma$,
denoted by $\alpha(\Gamma)$, is the maximum cardinality of an
independent set among all independent sets of $\Gamma$. For
convenience, we will denote $\alpha({\rm GK}(G))$ as $t(G)$ for a
group $G$. Moreover, for a vertex $r\in \pi(G)$, let $t(r,G)$
denote the maximal number of vertices in independent sets of ${\rm
GK}(G)$ containing $r$.

\begin{theorem}[Theorem 1, \cite{vg}]\label{vas-g}
Let $G$ be a finite group with $t(G)\geqslant 3$ and
$t(2,G)\geqslant 2$, and let $K$ be the largest normal solvable
subgroup $K$ of $G$. Then the quotient group $G/K$ is an almost
simple group, i.e., there exists a finite non-abelian simple
group $S$ such that $S\leqslant G/K\leqslant {\rm Aut}(S)$.
\end{theorem}

Given a prime $p$, we denote by $\mathcal{S}_p$ the set of
non-abelian finite simple groups $P$ such that $\max \pi(P)=p$.
Using Table 1 in \cite{za} (see also \cite[Table
4]{banoo-alireza}), we have listed the non-abelian simple groups
and their orders in $\mathcal{S}_{13}$, in Table 2.

\begin{center}
{\bf Table 2}. {\em The simple groups $S$ with
$\pi(S)\subseteq\{2, 3, 5, 7, 11, 13\}$.}\\[0.4cm]
\begin{tabular}{|llc||llc|}\hline
$S$ & $|S|$ & $|{\rm Out}(S)|$ & $S$ & $|S|$ & $|{\rm Out}(S)|$ \\[0.1cm]
\hline
&  & & & &  \\[-0.26cm]
$A_5$ & $2^2\cdot3\cdot5$ & $2$ & $A_{13}$ & $2^9 \cdot3^5
\cdot5^2 \cdot7 \cdot 11\cdot 13$ & $2$ \\ [0.1cm] $A_6$ &
$2^3\cdot3^2\cdot5$ &
$4$ & $A_{14}$ & $2^{10} \cdot3^5 \cdot5^2 \cdot7^2 \cdot 11\cdot 13$ & $2$\\[0.1cm]
$U_4(2)$& $2^6\cdot 3^4\cdot 5$ & $2$ & $A_{15}$ & $2^{10} \cdot3^6 \cdot5^3 \cdot7^2 \cdot 11\cdot 13$ & $2$ \\[0.1cm]
$A_7$ & $2^3\cdot3^2\cdot5\cdot7$ & $2$ & $A_{16}$ &
$2^{14} \cdot3^6 \cdot5^3 \cdot7^2 \cdot 11\cdot 13$ & $2$ \\
[0.1cm] $A_8$ & $2^6\cdot3^2\cdot5\cdot7$ & $2$ & $ G_2(2^2) $ &
$2^{12}\cdot 3^3\cdot 5^2\cdot 7\cdot 13$ & $2$ \\ [0.1cm] $A_9$ &
$2^6\cdot3^4\cdot5\cdot7$ & $2$ & $ B_2(2^3) $ & $2^{12}\cdot
3^4\cdot 5\cdot 7^2\cdot 13$ & $ 6$ \\ [0.1cm] $A_{10}$ &
$2^7\cdot3^4\cdot5^2\cdot7$ & $2$ & $ {\rm Sz}(2^3)$ &
$2^{6}\cdot 5\cdot 7\cdot 13$ & $3$ \\ [0.1cm]
$B_3(2)$ & $2^9\cdot 3^4\cdot 5\cdot 7$ & $1$ & $ L_2(2^6) $ & $2^{6}\cdot 3^2\cdot 5\cdot 7\cdot 13$ & $6$\\[0.1cm]
$O_8^+(2)$ & $2^{12}\cdot 3^5\cdot 5^2\cdot 7$ & $6$ & $ L_3(3) $ & $2^{4}\cdot 3^3\cdot 13$ & $2$\\[0.1cm]
$L_3(2^2)$ & $2^6\cdot 3^2\cdot 5\cdot 7$& ${12}$ & $ L_4(3) $ & $2^{7}\cdot 3^6 \cdot 5\cdot 13$ & $4$ \\[0.1cm]
$L_2(2^3)$ & $2^3\cdot 3^2\cdot 7$ & $3$ & $L_5(3)$ & $2^9\cdot 3^{10}\cdot 5 \cdot {11}^2\cdot 13$ & $2$ \\[0.1cm]
$U_3(3)$ & $2^5\cdot 3^3\cdot 7$ & $2$ & $L_6(3)$ & $2^{11}\cdot 3^{15}\cdot 5\cdot 7\cdot {11}^2\cdot {13}^2$ & $4$ \\[0.1cm]
$U_4(3)$ & $2^7\cdot 3^6\cdot 5\cdot 7$ & $8$ & $ B_3(3)$ & $2^{9}\cdot 3^9\cdot 5\cdot 7\cdot 13$ & $2$\\[0.1cm]
$U_3(5)$ & $2^4\cdot 3^2\cdot 5^3\cdot 7$ & $6$ & $ O_8^+(3)$ & $2^{12}\cdot 3^{12}\cdot 5^2\cdot 7\cdot 13$ & $24$ \\[0.1cm]
$L_2(7)$& $2^3\cdot 3\cdot 7$ &  $2$ & $ G_2(3) $ & $2^{6}\cdot 3^6\cdot 7\cdot 13$ & $2 $ \\[0.1cm]
$B_2(7)$ & $2^8\cdot 3^2\cdot 5^2\cdot 7^4$& $2$ & $ C_3(3)$ & $2^{9}\cdot 3^9\cdot 5\cdot 7\cdot 13$ & $2$\\[0.1cm]
$L_2(7^2)$ & $2^4\cdot 3\cdot 5^2\cdot 7^2$ & $4$ & $ L_3(3^2) $ & $2^{7}\cdot 3^6\cdot 5\cdot 7\cdot 13$ & $4$ \\[0.1cm]
$J_2$ & $2^7\cdot 3^3\cdot 5^2\cdot 7$ & $2$ & $ L_2(3^3) $ & $2^{2}\cdot 3^3\cdot 7\cdot 13$ & $6$\\[0.1cm]
$A_{11}$ & $2^7 \cdot3^4 \cdot5^2 \cdot7 \cdot 11$ & $2$ &
 $ U_4(5)$ &
$2^{7}\cdot 3^4\cdot 5^6\cdot 7\cdot 13$ & $4$ \\ [0.1cm]
$A_{12}$ & $2^9 \cdot3^5 \cdot5^2 \cdot7 \cdot 11$ & $2$ &
 $ B_2(5) $
& $2^{6}\cdot 3^2\cdot 5^4\cdot 13$ & $2 $ \\ [0.1cm] $U_5(2)$ &
$2^{10}\cdot3^5\cdot5\cdot11$ & $2$ &   $ L_2(5^2)$ & $2^{3}\cdot
3\cdot 5^2\cdot 13$ & $4$      \\ [0.1cm] $U_6(2)$ &
$2^{15}\cdot3^6\cdot5\cdot7\cdot11$ & $6$ &   $L_2(13)$ &
$2^{2}\cdot 3\cdot 7\cdot 13$ & $2$           \\
[0.1cm] $L_2(11)$ & $2^2\cdot3\cdot5\cdot11$ & $2$ &   $Suz$ &
$2^{13}\cdot 3^{7}\cdot 5^2\cdot 7\cdot 11\cdot 13$ & $2$     \\[0.1cm]
$M_{11}$ & $2^4\cdot3^2\cdot5\cdot11$ & $1$ &     $Fi_{22}$ & $2^{17}\cdot 3^{9}\cdot 5^2\cdot 7\cdot 11\cdot 13$ & $2$            \\
[0.1cm] $M_{12}$ & $2^6\cdot3^3\cdot5\cdot11$ & $2$ &     $ {^3D_4(2)}$ & $2^{12}\cdot 3^4\cdot 7^2\cdot 13$ & $3$        \\
[0.1cm] $M_{22}$
& $2^7\cdot3^2\cdot5\cdot7\cdot11$ & $2$ &  $ {^2F_4(2)}'$ & $2^{11}\cdot 3^3\cdot 5^2\cdot 13$ & $ 2$  \\[0.1cm]  $HS$ &
$2^{9}\cdot 3^2\cdot 5^3\cdot 7\cdot 11$ & $2$ &     $ U_3(2^2)$ &
$2^{6}\cdot 3\cdot 5^2\cdot 13$ & $4$                 \\[0.1cm]
$M^cL$ & $2^7\cdot 3^6\cdot
5^3\cdot 7\cdot 11$ & $2$&  &   &  \\[0.1cm]
\hline
\end{tabular}
\end{center}

\begin{lm}\label{inK} Let $G$ be a finite group, with
$|G|=p_1^{m_1}p_2^{m_2}\cdots p_s^{m_s}$, where $s, m_1,
m_2,\ldots, m_s$ are positive integers and $p_1, p_2, \ldots,
p_s$ distinct primes. Let $\Delta=\{p_i\in \pi(G) \ | \ m_i=1
\}$, and for each prime $p_i\in \Delta$, let
$\Delta(p_i)=\{p_j\in \Delta \ | \ j\neq i, \ p_j\nmid p_i-1
\mbox{and} \ p_i\nmid p_j-1\}$. Let $K$ be a normal slovable
subgroup of $G$. Then there hold.
\begin{itemize}
\item[$(1)$] If $p_i\in \Delta$ divides the
order of $K$, then for each prime $p_j\in \Delta(p_i)$,  $p_j\sim
p_i$ in ${\rm GK}(G)$. In particular, $\deg_G(p_i)\geqslant
|\Delta(p_i)|$.

\item[$(2)$] If $|\Delta|=s$ and for all $i=1, 2, \ldots, s$, $|\Delta(p_i)|=s-1$, then
$G$ is a cyclic group of order $|G|$ and ${\rm D}(G)=(s-1, s-1,
\ldots, s-1)$.
\end{itemize}
\end{lm}
\begin{proof} see \cite{akbari-1}.  \end{proof}

\begin{theorem}
The projective special linear group $L_6(3)$ is ${\rm
OD}$-characterizable.
\end{theorem}
\begin{proof}
Assume that $G$ be a finite group such that
$$|G|=|L_6(3)|=2^{11}\cdot 3^{15}\cdot 5\cdot 7\cdot {11}^2\cdot {13}^2
\ \ \ \mbox{and} \ \ \ {\rm D}(G)={\rm D}(L_6(3))=(4, 3, 2, 2, 0,
3).$$ According to these conditions on $G$, the following two
possibilities can occur for the prime graph of $G$:

\vspace{0.95cm}

{\small  \setlength{\unitlength}{4mm}
\begin{picture}(1,1)(4,14)
\linethickness{0.3pt} %
\put(11,14){\circle*{0.5}}
\put(16,14){\circle*{0.5}}
\put(13.5,11.5){\circle*{0.5}}
\put(18.5,11.5){\circle*{0.5}}
\put(8.5,11.5){\circle*{0.5}}
\put(8.5,11.5){\line(1,1){2.5}}
\put(11,14){\line(1,0){5}}
\put(16,14){\line(1,-1){2.5}}
\put(11,14){\line(1,-1){2.5}}
\put(13.5,11.5){\line(1,1){2.5}}
\put(8.2,10.2){\small$5$}%
\put(10.7,14.7){\small$3$}%
\put(15.7,14.7){\small$13$}%
\put(18.2,10.2){\small$7$}%
\put(13.3,10.2){2}%
\put(13.5,16){\circle*{0.5}}
\put(13.2,16.5){\small$11$}%
\put(8.5,11.5){\line(1,0){10}}
\put(23,12.75){\rm or}%
\put(30,14){\circle*{0.5}}
\put(35,14){\circle*{0.5}}
\put(32.5,11.5){\circle*{0.5}}
\put(37.5,11.5){\circle*{0.5}}
\put(27.5,11.5){\circle*{0.5}}
\put(27.5,11.5){\line(1,1){2.5}}
\put(30,14){\line(1,0){5}}
\put(35,14){\line(1,-1){2.5}}
\put(30,14){\line(1,-1){2.5}}
\put(32.5,11.5){\line(1,1){2.5}}
\put(27,10.2){\small$5$}%
\put(29.5,14.7){\small$13$}%
\put(34.5,14.7){\small$3$}%
\put(37,10.2){\small$7$}%
\put(32.1,10.2){2}%
\put(32.5,16){\circle*{0.5}}
\put(32.2,16.5){\small$11$}%
\put(27.5,11.5){\line(1,0){10}}
\put(13,8){{\small \bf Fig. 3.} \mbox{All possibilities for the
prime graph} ${\rm GK}(G)$.}
\end{picture}}
\vspace{3.2cm}

\noindent Let $K$ be the largest normal solvable subgroup of $G$.
We claim that $K$ is a $\{2, 3\}$-group. To prove this, we first
observe that $\pi(K)\cap \{5, 7\}=\emptyset$, since otherwise by
Lemma \ref{inK}, $5\sim 7$, which clashes with the structure of
${\rm GK}(G)$ shown in Fig. 3. Now, we show that $\pi(K)\cap \{11,
13\}=\emptyset$. Suppose the contrary. Assume first that
$11\in\pi(K)$ and take $P\in{\rm Syl}_{11}(K)$. Then, by Frattini
argument, we deduce that $G=KN_G(P)$. Therefore, the normalizer
$N_G(P)$ contains an element of order $7$, say $x$, and so
$P\langle x\rangle$ is a subgroup of $G$ of order $11^i\cdot 7$,
$i=1, 2$, which is a cyclic group in both cases. This shows that
$7\sim 11$ in ${\rm GK}(G)$, an impossibility. By a similar way
we can show that $13\notin\pi(K)$. Finally $K$ is a $\{2,
3\}$-group.

From the structure of the prime graph of $G$ as shown in Fig. 3,
we conclude that $t(2, G)\geq 2$ and $t(G)\geq 3$. Consequently,
from Theorem \ref{vas-g} we conclude that there exists a finite
non-abelian simple group $S$ such that $S\leq G/K\leq {\rm
Aut}(S)$. Since $S\in {\cal S}_{13}$ and $K$ is a $\{2,
3\}$-group, we observe that the order of ${\rm Aut}(S)$ is
divisible by $5\cdot7\cdot {11}^2\cdot {13}^2$. On the other
hand, using Table 2 we see that $\pi({\rm Out}(S))\subseteq\{2,
3\}$, which forces $|S|=2^a\cdot 3^b\cdot 5\cdot 7\cdot
{11}^2\cdot {13}^2$, where $2\leq a\leq 11$ and $0\leq b\leq 15$.
But then, Table 2 shows that the only possibility for $S$ is
$L_6(3)$ and since $|G|=|L_6(3)|$, we obtain $|K|=1$ and $G$ is
isomorphic to $L_6(3)$. \end{proof}
\begin{theorem}
The projective special unitary group $U_4(5)$ is ${\rm
OD}$-characterizable.
\end{theorem}
\begin{proof}
Assume that $G$ is a finite group such that
$$|G|=|U_4(5)|=2^7\cdot 3^4\cdot 5^6\cdot 7\cdot {13}
\ \ \ \mbox{and} \ \ \ {\rm D}(G)={\rm D}(U_4(5))=(3, 3, 2, 1,
1).$$ We have to show that $G\cong U_4(5)$. First of all, from the
structure of the degree pattern of $G$, it is easy to see that
$7\nsim 13$ in ${\rm GK}(G)$, since otherwise ${\rm deg}(2)\leq
2$, which is impossible. In fact, there are only two
possibilities for the prime graph of $G$ shown in Fig. 4.:

\vspace{0.5cm}

{\small  \setlength{\unitlength}{4mm}
\begin{picture}(1,1)(4,14)
\linethickness{0.3pt} %
\put(11,14){\circle*{0.5}}
\put(16,14){\circle*{0.5}}
\put(13.5,11.5){\circle*{0.5}}
\put(18.5,11.5){\circle*{0.5}}
\put(8.5,11.5){\circle*{0.5}}
\put(8.5,11.5){\line(1,1){2.5}}
\put(11,14){\line(1,0){5}}
\put(16,14){\line(1,-1){2.5}}
\put(11,14){\line(1,-1){2.5}}
\put(13.5,11.5){\line(1,1){2.5}}
\put(8.2,10.2){\small$13$}%
\put(10.7,14.7){\small$3$}%
\put(15.7,14.7){\small$2$}%
\put(18.2,10.2){\small$7$}%
\put(13.3,10.2){5}%
\put(23,12.75){\rm or}%
\put(30,14){\circle*{0.5}}
\put(35,14){\circle*{0.5}}
\put(32.5,11.5){\circle*{0.5}}
\put(37.5,11.5){\circle*{0.5}}
\put(27.5,11.5){\circle*{0.5}}
\put(27.5,11.5){\line(1,1){2.5}}
\put(30,14){\line(1,0){5}}
\put(35,14){\line(1,-1){2.5}}
\put(30,14){\line(1,-1){2.5}}
\put(32.5,11.5){\line(1,1){2.5}}
\put(27,10.2){\small$13$}%
\put(29.5,14.7){\small$2$}%
\put(34.5,14.7){\small$3$}%
\put(37,10.2){\small$7$}%
\put(32.1,10.2){5}%
\put(13,8){{\small \bf Fig. 4.} \mbox{All possibilities for the
prime graph} ${\rm GK}(G)$.}
\end{picture}}
\vspace{3.2cm}

Since $S\in{\cal S}_{13}$ and $\{7, 13\}\subseteq \pi(S)$, hence
we obtain $|S|=2^a\cdot 3^b\cdot 5^c\cdot 7\cdot 13$, where $2\leq
a\leq 7$, $0\leq b\leq 4$ and $0\leq c\leq 6$. Comparing the
orders of simple groups listed in Table 2, we observe that the
only possibility for $S$ is $L_2(3^3)$ or $U_4(5)$. If $S$ is
isomorphic to $L_2(3^3)$, then $|K|$ is divisible by $5^6$,
because $|{\rm Out}(S)|=4$. Let $x\in K$ be an element of order 5
and $P\in{\rm Syl}_{5}(K)$. By Frattini argument, we deduce that
$G=KN_G(P)$. Hence the normalizer $N_G(P)$ contains an element of
order $13$, say $y$, and so $H:=P\langle y\rangle$ is a subgroup
of $G$ of order $5^6\cdot 13$. Now, one can easily verify that
the Sylow 5-subgroup $P$ of $H$ is a normal subgroup of $H$. As a
matter of fact, $H$ is a Frobenius group with kernel $P$ and
complement $\langle y\rangle$, because $H$ does not contain an
element of order $5\cdot 13$. But then we must have $13\mid
5^6-1$, which is a contradiction. Therefore, $S$ is isomorphic to
$U_4(5)$, and since $|G|=|U_4(5)|$, we obtain $|K|=1$ and $G$ is
isomorphic to $U_4(5)$. \end{proof}

\end{document}